\documentclass[a4paper, reqno]{amsart}
\usepackage[a4paper, margin=1.5in]{geometry}
\usepackage[utf8]{inputenc}
\usepackage[english]{babel}

\usepackage{graphicx}
\usepackage[pdfborder={0 0 0}]{hyperref}
\usepackage[utf8]{inputenc}
\usepackage[T1]{fontenc}
\usepackage{microtype}

\usepackage[mathscr]{eucal}

\usepackage[shortlabels]{enumitem}
\usepackage[super]{nth}

\usepackage[textsize = footnotesize]{todonotes}

\usepackage{amsmath}
\usepackage{amssymb}

\usepackage{amsthm}
\usepackage{graphicx}
\usepackage{ bbold }
\usepackage[capitalize, nameinlink]{cleveref}
\crefformat{equation}{#2(#1)#3}
\Crefformat{equation}{#2(#1)#3}
\crefname{section}{Chapter}{Chapters}
\crefname{subsection}{Section}{Sections}
\crefname{subsubsection}{Section}{Sections}
\crefname{subappendix}{Section}{Sections}
\crefname{subsubappendix}{Section}{Sections}

\usepackage{mathtools}
\usepackage{xparse}
\usepackage{tikz-cd}

\theoremstyle{plain}

\newtheorem{theorem}{Theorem}[section]
\newtheorem{lemma}[theorem]{Lemma}
\newtheorem{prop}[theorem]{Proposition}
\newtheorem{corollary}[theorem]{Corollary}

\theoremstyle{definition}
\newtheorem{remark}[theorem]{Remark}
\newtheorem{definition}[theorem]{Definition}
\newtheorem{example}[theorem]{Example}

\theoremstyle{remark}

\DeclareMathOperator{\colim}{colim}

\DeclareMathOperator{\ev}{ev}

\DeclareMathOperator{\Alg}{Alg}

\DeclareMathOperator{\Map}{Map}

\newcommand{\hide}[1]{}

\newcommand{\op}{\mathrm{op}}
\newcommand{\Fun}{\mathrm{Fun}}
\newcommand{\Fin}{\mathrm{Fin}}
\newcommand{\sC}{\mathscr{C}}
\newcommand{\sD}{\mathscr{D}}
\newcommand{\sE}{\mathscr{E}}
\newcommand{\sP}{\mathscr{P}}
\newcommand{\sO}{\mathscr{O}}

\newcommand{\Spc}{\mathscr{S}}

\newcommand{\Cat}{\mathscr{C}\mathrm{at}}

\newcommand{\Env}{\mathrm{Env}}

\newcommand{\Z}{\mathbf{Z}}

\newcommand{\E}{\mathbf{E}}

\NewDocumentCommand\derprojlim{e{_}}{\mathchoice
{\varprojlim  \IfValueT{#1}{_{\mathclap{#1}}}{}^{\!1\!}\mathop{}}
{\varprojlim^1  \IfValueT{#1}{_{#1}}}
{\varprojlim^1  \IfValueT{#1}{_{#1}}}
{\varprojlim^1  \IfValueT{#1}{_{#1}}}}

\newcommand{\Op}{\mathrm{Op}}
\newcommand{\Mon}{\mathrm{Mon}}
\newcommand{\lax}{\mathrm{lax}}
\newcommand{\Act}{\mathrm{Act}}

\newcommand{\Com}{\mathrm{Com}}
\newcommand{\Triv}{\mathrm{Triv}}
\newcommand{\1}{\mathbb{1}}

\title{Free Algebras via Monoidal Envelopes}

\author{Max Blans}
\address{Mathematical Institute, University of Oxford}
\email{max.blans@maths.ox.ac.uk}

\author{Sil Linskens}
\address{Fakult\"at f\"ur Mathematik, University of Regensburg}
\email{sil.linskens@ur.de}

\begin{document}

\begin{abstract}
 For any morphism of $\infty$-operads $\sP \to \sO$, we show that the free $\sO$-algebra on a $\sP$-algebra admits an explicit formula as the colimit over the $\sO$-monoidal envelope of $\sP$, providing a new and simple proof of the existence of relative free $\sO$-algebras.
\end{abstract}

\maketitle
\tableofcontents

\section{Introduction}
The study and use of homotopy coherent algebraic structure has a long history. In this new millennium, the methods by which we study such structures have been revolutionized by the systematic use of higher categories, as developed by \cite{HTT}, among many others. Specifically, Lurie introduced in \cite{HA} the theory of $\infty$-operads, which have become the most popular framework for studying homotopy coherent algebraic structures.

Crucial to any algebraic theory is an understanding of free algebras. It is of course well-known that given a (single colored) $\infty$-operad $\sO$ and a cocomplete symmetric monoidal $\infty$-category $\sC$ for which the tensor product commutes with colimits in each variable separately, the free $\sO$-algebra in $\sC$ on an object $X\in \sC$ is computed by the formula
\[
\mathrm{free}_{\sO}(X) \simeq \coprod_{n\geq 0} (\sO(n)\otimes X^{\otimes n})_{h \Sigma_n}.
\]
If the tensor product does not commute with all colimits then free algebras may still exist, but the formula is in general more complicated. A typical example is the category of pointed spaces equipped with the cartesian product. In this case the product only commutes with colimits over weakly contractible 
categories, and so the previous formulas are no longer correct. However there is a replacement: for example the free associative algebra on a pointed space $X$ is given by the \emph{James construction}, the colimit over the diagram 
\[\begin{tikzcd}
	X & {X\times X} & {X\times X\times X} & \dots,
	\arrow[shift left, from=1-1, to=1-2]
	\arrow[shift right, from=1-1, to=1-2]
	\arrow[shift left=2, from=1-2, to=1-3]
	\arrow[shift right=2, from=1-2, to=1-3]
	\arrow[from=1-2, to=1-3]
	\arrow[shift left, from=1-3, to=1-4]
	\arrow[shift right=3, from=1-3, to=1-4]
	\arrow[shift left=3, from=1-3, to=1-4]
	\arrow[shift right, from=1-3, to=1-4]
\end{tikzcd}\]
where the maps are given by inclusion of the basepoint into each factor. More precisely the indexing category of the colimit above is the category $\Delta^{\mathrm{inj}}$, the subcategory of the simplex category on the injective maps. In this article we are concerned with understanding how this and other similar formulas arise, and in which generality they compute free algebras. 

More specifically, we may reinterpret the construction before as an instance of a relative free algebra construction. As the most general situation we consider a map $p\colon \sP\to \sO$ of $\infty$-operads. We also suppose that $\sO$ has one color.  This assumption is not necessary, but we suppose it in this introduction for simplicity. Finally let $\sC$ be an $\sO$-monoidal $\infty$-category. Then we may consider the restriction functor 
\[
\mathrm{fgt}_\sP^\sO\colon \Alg_{\sO}(\sC) \to \Alg_{\sP}(\sC),
\]
which takes an $\sO$-algebra in $\sC$ and restricts it to a $\sP$-algebra. In the extreme case that $\sP$ is the trivial operad we obtain the forgetful functor $\Alg_{\sO}(\sC)\to \sC$. If $p$ is the map $\E_0\to 
\E_1$ and $\sC$ is the $\infty$-category of pointed spaces, we recover the forgetful functor $\Alg_{\E_1}(\Spc_\ast) \to \Spc_{\ast}$ considered before. The question addressed by this article is: When does the functor $\mathrm{fgt}_\sP^\sO$ admit a left adjoint $\mathrm{free}^{\sO}_{\sP}$?  We may call this supposed left adjoint the \emph{relative free functor},\footnote{This left adjoint is called \emph{operadic left Kan extension} in \cite{HA}.} and its value on $X\in \Alg_{\sP/\sO}(\sC)$ \emph{the relative free algebra on $X$}. Just as important is the next question: how can we effectively provide a formula for these relative free functors?

Both these questions are addressed in \cite{HA}. However the construction is very delicate, and specific to the point-set implementation of $\infty$-operads given there. It is also not so easy to read off the precise conditions under which a particular relative free algebra exists, or how to compute it explicitly. In this short article we observe that, under relatively mild conditions, relative free algebras can be constructed without much hard work. Moreover, our construction also immediately provides a formula for them.

We continue to use the notation introduced above. Given the map $p\colon \sP \to \sO$ of $\infty$-operads, we may consider the \emph{$\sO$-monoidal envelope} $\Env_{\sO}(\sP)$ of $\sP$, the universal $\sO$-monoidal $\infty$-category equipped with a $\sP$-algebra. In particular, given a $\sP$-algebra $A\colon \sP\to \sC$ in $\sC$, we obtain an $\sO$-monoidal functor 
\[
F_A\colon \Env_{\sO}(\sP)\to \sC.
\] 
Supposing that the tensor product $\otimes_f \colon \sC^{\times n} \to \sC$ commutes with $\Env_{\sO}(\sP)$ shaped colimits for every active arrow $f$ of $\sO$, we show that the colimit of $F_A$ is canonically a $\sO$-algebra with the universal property of the relative free algebra of $A$ along $p$. In other words, we show 
\[
\mathrm{free}_{\sP}^{\sO}(A) \simeq \colim\limits_{\Env_{\sO}(\sP)} F_A.
\]
If $p$ is the inclusion $\E_0\to \E_1$ then $\Env_{\E_1}(\E_0)$ is precisely $\Delta^\mathrm{inj}$, and so we deduce that the James construction computes the free associative algebra on an $\E_0$-algebra in large generality. In particular it applies to any symmetric monoidal category $\sC$ in which the tensor product commutes with weakly contractible colimits in each variable. An example of such a $\sC$ is the category of spectral Lie algebras. We use this to deduce that in spectral Lie algebras, there is an equivalence between $\Omega \Sigma X$ and the James construction on $X$.

Beyond simply proving the existence of relative free algebras, we also emphasize the utility of the formula obtained above. For example we specialize this in different cases to immediately obtain well-known explicit formulas for specific instances of (relative) free algebras. In another direction we explain how to construct and compute colimits in $\infty$-categories of algebras using the formula above.

\subsection*{Relation to other work}

The relation between monoidal envelopes and the construction of free algebras has appeared before, and is present implicitly in much recent work. In fact, our main result has appeared explicitly in the work of Nardin and Shah \cite{NS}, who moreover prove our main result in the more generalized setting of $\infty$-operads parametrized by an atomic orbital $\infty$-category. 
Morally our two approaches are similar, but the details of the proof are different. We hope that this article serves as an advertisement of the utility of this perspective on free algebras by emphasizing both the simplicity of the main result and its proof, as well as its wide applicability.

\subsection*{Conventions}
For the entirety of the article we will fix an arbitrary $\infty$-operad $\sO$. Given an $\infty$-operad $\sP$ over $\sO$ and an $\sO$-monoidal $\infty$-category $\sC$, we will denote by $\Alg_{\sP}(\sC)$ the category of $\sP$-algebras in $\sC$. This is denoted $\Alg_{\sP/\sO}(\sC)$ in \cite{HA}. We will typically omit the $\otimes$-superscript often appended to $\sO$-operads and $\sO$-monoidal $\infty$-categories. However if we wish to emphasize the existence of either as a fibration over $\Fin_\ast$, we may include the superscript. We will call a functor $f \colon I \to J$ of $\infty$-categories final if for every diagram $X \colon J \to \sC$ to an $\infty$-category $\sC$, the comparison map $\colim_I(X \circ f) \to \colim_J X$ is an isomorphism.
We say an $\infty$-category $\sC$ is weakly contractible if the space $|\sC|$ obtained by inverting all its morphisms is contractible.
Finally, since we work exclusively with $\infty$-categories and $\infty$-operads, we will drop the $\infty$ and simply refer to them as categories and operads.

\subsection*{Acknowledgements}
The authors would like to thank the organizers of the Young Topologist Meeting 2022 and the University of Copenhagen, where this collaboration began.
They would also like to thank Gijs Heuts and Rune Haugseng for helpful discussions related to this article.

While working on this project the first author was supported by the European Research Council (ERC) through the grant ``Chromatic homotopy theory of spaces'' (no. 950048) and by the Royal Society through the grant URF\textbackslash R1\textbackslash 211075.
The second author was supported by DFG Schwerpunktprogramm 1786 “Homotopy Theory and Algebraic Geometry” (project ID SCHW 860/1-1). 

\section{Some universal constructions on operads}

In this section we recall some basic universal constructions on operads which will be necessary for our construction of free algebras.

\subsection{Envelopes of $\sO$-operads}

\begin{definition}
Let $\Op$ be the category of operads and write $\Op_\sO$ for $\Op_{/\sO}$.
We define $\Mon^{\lax}_\sO \subset \Op_{\sO}$ to be the full subcategory of $\Op_{\sO}$ spanned by the $\sO$-monoidal categories. We define $\Mon_{\sO} \subset \Mon^{\lax}_{\sO}$ to be the wide subcategory spanned by the strong $\sO$-monoidal functors.
\end{definition}

\begin{lemma}\label{lem:Univ_Prop_Env}
The inclusion 
\[
\Mon_{\sO} \hookrightarrow \Op_{\sO}
\]
admits a left adjoint, which we denote by 
\[
\Env_{\sO}(-)\colon \Op_{\sO} \rightarrow \Mon_{\sO}.
\]
Given an operad $\sP$ over $\sO$, the $\sO$-monoidal category $\Env_{\sO}(\sP)$ is given by the fiber product
\[\begin{tikzcd}
	{\Env_{\sO}(\sP)^\otimes} & \sP^\otimes \\
	{\Act(\sO^\otimes)} & \sO^\otimes \\
	{\sO^\otimes,}
	\arrow[from=1-2, to=2-2]
	\arrow["{\ev_0}"', from=2-1, to=2-2]
	\arrow["{\ev_1}", from=2-1, to=3-1]
	\arrow[from=1-1, to=2-1]
	\arrow[from=1-1, to=1-2]
	\arrow["\lrcorner"{anchor=center, pos=0.125}, draw=none, from=1-1, to=2-2]
\end{tikzcd}\]
 which lives over $\sO^{\otimes}$ via the long vertical composite. Finally $\Act(\sO^{\otimes})$ is defined to be the full subcategory of $\Fun([1],\sO^{\otimes})$ spanned by the active arrows.
\end{lemma}

\begin{proof}
See \cite[Section 2.2.4]{HA}, or \cite{BHS}.
\end{proof}

\begin{remark}\label{rem: univ-alg-env}
Let $\sP$ be an operad over $\sO$. The unit of adjunction above gives a map
\[
\eta \colon \sP \to \Env_{\sO}(\sP),
\]
of $\sO$-operads, in other words a $\sP$-algebra in $\Env_{\sO}(\sP)$. 
If $\sC$ is any $\sO$-monoidal category and $A\colon \sP \to \sC$ is a $\sP$-algebra in $\sC$, then the adjunction provides an essentially unique $\sO$-monoidal functor
\[
F_A \colon \Env_\sO(\sP) \to \sC
\]
such that $F_A \circ \eta \simeq A$. Therefore, $\Env_{\sO}(\sP)$ is the universal $\sO$-monoidal category with a $\sP$-algebra.
\end{remark}

\begin{example} \label{ex: env-e0-com}
We write $\E_\infty = \Fin_*$ for the commutative operad. We also denote by $\E_0 = \Fin_*^{\mathrm{inj}} \subset \Fin_*$ the suboperad spanned by maps of finite pointed sets which are injective away from the basepoint. We recall that the category of $\E_0$-algebras in a monoidal category $\sC$ is equivalent to $\sC_{\1/}$, the category of objects under the unit $\1$ of $\sC$. A simple computation shows that
\[\Env_{\E_\infty}(\E_0) \simeq \Fin^{\mathrm{inj}},\]
the category of finite sets and injections equipped with the coproduct symmetric monoidal structure. 
\end{example}

\begin{example} \label{ex: env-e0-e1}
Let $\E_1$ denote the associative operad. The reader may wish to recall the explicit description of $\E_1$ given in \cite[Remark 4.1.1.4]{HA}. Once again a simple computation shows that
\[
\Env_{\E_1}(\E_0) \simeq \Delta_{+}^{\mathrm{inj}},
\] 
the augmented semi-simplex category equipped with the concatenation monoidal structure.
\end{example}

\begin{example}\label{ex:Env_of_triv}
Suppose $\sO$ is a single colored unital operad in the sense of \cite[Definition 2.3.1.1]{HA}, meaning that $\sO$ is pointed as a category. Let $\Triv\subset \Fin_\ast$ denote the suboperad spanned by the inert maps. One computes that
\[\Env_{\sO}(\sO \times \Triv) \simeq \coprod \sO(n)_{h\Sigma_n}.\]
In particular, for $\sO = \E_\infty$ we find that \[\Env_{\E_\infty}(\Triv) \simeq \coprod \ast_{h\Sigma_n} = \coprod B\Sigma_n = \Fin^\simeq.\]
\end{example}

\begin{example}
One may compute that $\Env_{\E_\infty}(\E_1)$ is given by the 1-category whose objects are given by finite sets $n,m$, such that a morphism $f\colon I\to J$ is given by a map of sets from $n$ to $m$, together with a linear order of the pre-image of each $j \in J$. This is symmetric monoidal via concatenation of finite sets. 
\end{example}

\begin{example}
Let $\E_k$ be the little $k$-cubes operad. As discussed in \cite[Remark 5.1.0.5.]{HA}, $\Env_{\E_\infty}(\E_k)$ can be identified with the topological category whose objects are given by disjoint unions of $k$-cubes and whose morphisms are embeddings which are rectilinear on each component. This is symmetric monoidal via disjoint union.
\end{example}

\begin{example}
More generally, given $k<j$ one can exhibit $\E_k$ as an $\E_j$-operad. Then $\Env_{\E_j}(\E_k)$ can be identified with the topological category whose objects are given by disjoint unions of $k$-cubes with a chosen component-wise rectilinear embedding into the first $k$-components of a $j$-cube. Morphisms are given by component-wise rectilinear embeddings which make the obvious triangle commute. This is $\E_j$-monoidal by whiskering embeddings into the $j$-cube in the natural way.
\end{example}

\subsection{Slices of $\sO$-monoidal categories}

In this section we recall the monoidality of slice categories.

\begin{definition}
Let $\sC$ be an $\sO$-monoidal category. Consider an $\sO$-algebra $A$ in $\sC$, given by an operad map $A \colon \sO\rightarrow \sC$. We define $\sC_{/A}$ to be the pullback 
\[\begin{tikzcd}
	{\sC_{/A}} & \sC^{[1]} \\
	\sO & \sC
	\arrow["t", from=1-2, to=2-2]
	\arrow["A"', from=2-1, to=2-2]
	\arrow[from=1-1, to=1-2]
	\arrow[from=1-1, to=2-1]
\end{tikzcd}\] in $\Op_{\sO}$, where $\sC^{[1]}$ refers to the cotensor in $\Op_{\sO}$.
\end{definition}

\begin{remark}
We note that cotensors and pullbacks in $\Op_{\sO}$ are computed in $\Cat_{/\sO^\otimes}$. In particular, for any color $X$ of $\sO$, the category $(\sC_{/A})_X$ agrees with the slice of $\sC_X$ over the functor $A_X \colon \sO_X\rightarrow \sC_X$.
\end{remark}

\begin{lemma}[{\cite[Section 2.2.2]{HA}}]
Let $\sC$ be $\sO$-monoidal, and let $A$ be an algebra object in $\sC$. Then $\sC_{/A}$ is an $\sO$-monoidal category and the composite \[\sC_{/A}\hookrightarrow \sC^{[1]} \xrightarrow{s} \sC\] is strong $\sO$-monoidal.\qedhere
\end{lemma}

\begin{remark}
Suppose for simplicity that $\sO$ has a single colour. Then the $\sO$-monoidal structure on $\sC_{/A}$ constructed before is such that given an active arrow $\phi\colon \langle 2 \rangle  \rightarrow \langle 1\rangle$ in $\sO$ the corresponding tensor operation $\otimes_\phi\colon (\sC_{/A})^2\rightarrow \sC_{/A}$ is given by \[(X\rightarrow A, Y\to A) \mapsto X\otimes_\phi Y \to A\otimes_\phi A\xrightarrow{\phi_A} A.\] This discussion readily generalizes to higher-arity operations and operads of multiple colours.
\end{remark}

\begin{lemma}\label{univ-prop-slice}
Let $F \colon \sP \rightarrow \sC$ be a map of operads over $\sO$.
Then the space of lax monoidal lifts of $F$ along the projection $s \colon \sC_{/A} \to \sC$
\[\begin{tikzcd}
	& {\sC_{/A}} \\
	\sP & \sC
	\arrow["F"', from=2-1, to=2-2]
	\arrow["\widetilde{F}", dotted, from=2-1, to=1-2]
	\arrow[from=1-2, to=2-2]
\end{tikzcd}\]
is equivalent to the space $\Map_{\Alg_{\sP}(\sC)}(F,\mathrm{fgt}_{\sP}^{\sO}A)$ of $\sP$-algebra maps from $F$ to the composite
\[
\mathrm{fgt}_{\sP}^{\sO}A \colon \sP \to \sO \xrightarrow{A} \sC.
\]
\end{lemma}

\begin{proof}
This follows immediately after unwinding the universal property of the pullback and cotensor in $\Op_{\sO}$. See also \cite[Lemma 2.12]{AC-B}.
\end{proof}

\begin{remark}\label{rmk_tilde_F_strong}
Note that the forgetful functor $\sC_{/A}\rightarrow \sC$ is conservative. This implies that, in the setting above, if $F$ is strong $\sO$-monoidal then any lax $\sO$-monoidal lift $\widetilde{F}$ of $F$ to $\sD_{/A}$ is again strong $\sO$-monoidal.
\end{remark}

\section{Colimits of lax \texorpdfstring{$\sO$}{O}-monoidal functors}

Let $\sO$ be an operad. For any sequence of objects $X_1, \ldots, X_n, Y \in \sO$, we write $\mathrm{Mul}_{\sO}(\{X_{i}\}_{i = 1}^n, Y)$ for the space of multimorphisms from the sequence $X_1, \ldots, X_n$ to $Y$.

\begin{definition}
Let $\mathcal{I}$ be a category and let $\sC$ be an $\sO$-monoidal category. We say that $\sC$ is $\sO$-monoidally compatible with $\mathcal{I}$-shaped colimits if for each $X \in \sO$, $\sC_X$ admits 
$\mathcal{I}$-shaped colimits, and for every
sequence of objects $X_1, \ldots, X_n, Y \in \sO$ and every active morphism $f \in \mathrm{Mul}_{\sO}(\{X_i\}_{i = 1}^n, Y)$, the induced functor
\[
\otimes_f \colon \sC_{X_1} \times \cdots \times \sC_{X_n} \to \sC_Y
\]
commutes with $\mathcal{I}$-shaped colimits in each variable separately.
\end{definition}

\begin{lemma}\label{univ_prop_colim}
Let $\sC$ and $\sD$ be $\sO$-monoidal categories and suppose that for each $X \in \sO$, $\sD$ is $\sO$-monoidally compatible with $\sC_X$-shaped colimits. Suppose $F\colon \sC \rightarrow \sD$ is a lax $\sO$-monoidal functor. Then the colimit of $F$ is canonically an $\sO$-algebra in $\sD$, and as such enjoys the following universal property:
there is a natural equivalence
\[
\Alg_{\sO}(\colim F, A)\simeq \Alg_{\sC}(F, \mathrm{const}_A) 
\] between the space of $\sO$-algebra maps $\colim F \rightarrow A$ and the space of $\sC$-monoidal natural transformations $F\to \mathrm{const}_A$, where $\mathrm{const}_A$ denotes the composite
\[
\sC \to \sO \xrightarrow{A} \sD.
\]
\end{lemma}

\begin{proof}
Note that we may assume that both $\sC$ and $\sD$ are small. We moreover claim that $\sD$ embeds in an $\sO$-monoidal category which is presentably $\sO$-monoidal, as defined for example in \cite[Definition 3.26]{LNP}, such that the embedding preserves $\sC_X$-colimits for all $X\in \sC$. Namely we may simply take the composite 
\[
\sD \xrightarrow{\mathcal{Y}} \sP(\sD) \to L^{-1}\sP(\sD),
\] 
of the Yoneda embedding to the presheaf category on $\sD$, equipped with the Day convolution symmetric monoidal structure, followed by the localization onto the $L$-local objects, where $L$ is the collection of maps 
\[
\colim_{\sC_{X}}(\mathcal{Y}(G))\to \mathcal{Y}(\colim_{\sC_{X}}(G))
\]
for all $X\in \sO$ and $G\colon \sC_X\to \sD$. By \cite[Proposition 2.2.1.9.]{HA} the collection of $L$-local objects $L^{-1}\sP(\sD)$ is again presentably $\sO$-monoidal. Finally, by construction the composite $\sD \to L^{-1}\sP(\sD)$ preserves $\sC_X$-indexed colimits for all $X\in \sO$. 

In this way we reduce to the case where the $\sO$-monoidal category $\sD$ is compatible with all colimits. We then consider $\sO$-monoidal Day convolution operad $\Fun_\sO(\sC,\sD)$, which is an $\sO$-monoidal category. One can fairly easily see that the diagonal functor $\Delta \colon \sD\rightarrow \Fun_\sO(\sC,\sD)$ is an operadic right adjoint, see for example \cite[Proposition 3.34, Remark 3.35]{LNP}. In particular, passing to left adjoints we obtain an $\sO$-monoidal structure on the left adjoint
\[\colim_\sC \colon \Fun_\sO(\sC,\sD) \rightarrow \sD\]
of the diagonal functor. Therefore $\colim_\sC$ again induces a left adjoint functor on categories of $\sO$-algebras
\[\colim_\sC\colon \Fun^{\sO\textup{-lax}}(\sC,\sD) \simeq \Alg_\sO(\Fun_\sO(\sC,\sD))\rightarrow \Alg_\sO(\sD),\] where we have also applied the universal property of Day convolution for the first equivalence, proving the lemma.
\end{proof}

\begin{remark}
The result above is proved under (mildly) stronger hypothesis in {\cite[Section 2.3]{AC-B}}. We note that the proof there uses the theory of operadic colimits due to Lurie, while we avoid this by appealing to the universal property of Day convolution.
\end{remark}

\section{Free algebras}

Having recapped the main universal constructions necessary, we may introduce our main theorem.

\begin{theorem} \label{thm: free-algebra-colim-env}
Consider $\sP\in \Op_{\sO}$ and $\sC\in \Mon_{\sO}$. Suppose that $\sC$ is $\sO$-monoidally compatible with $\Env_\sO(\sP)$-shaped colimits. Then the forgetful functor \[\mathrm{fgt}^\sO_\sP \colon \Alg_{\sO}(\sC) \to \Alg_{\sP}(\sC)\] admits a left adjoint \[\mathrm{free}_{\sP}^{\sO} \colon \Alg_{\sP}(\sC) \to \Alg_{\sO}(\sC).\] Furthermore, the left adjoint is equivalent to the composite
\[\Alg_{\sP}(\sC)\simeq \Fun^{\sO\textup{-}\otimes}(\Env_\sO(\sP),\sC) \xrightarrow{\colim} \Alg_{\sO}(\sC).\] 
\end{theorem}

\begin{remark}
Fix a $\sP$-algebra $A \in \Alg_{\sP / \sO}(\sC)$ and let $F_A \colon \Env_{\sO}(\sP) \to \sC$ be the corresponding $\sO$-monoidal functor, as defined in Remark \ref{rem: univ-alg-env}. Then we may informally summarize the conclusion of the theorem as giving a natural equivalence
\[
\mathrm{free}_{\sP}^{\sO} A  \simeq \colim_{\Env_{\sO}(\sP)} F_A,
\] where the colimit of $F_A$ is canonically an $\sO$-algebra via Lemma \ref{univ_prop_colim}. 

We note that the theorem above also provides a description of the multiplication in $\mathrm{free}_{\sP}^{\sO} A$ associated to an active morphism in $\sO$, which is often used in the literature. For simplicity we assume $\sO$ has a single color and the active morphism is two-to-one, but the description clearly generalizes. Let us also write $\mathcal{E}$ for $\Env_\sO(\sP)$. In this case, given an active arrow $\phi\colon \langle 2\rangle \to \langle 1\rangle$, the corresponding operation 
\[
{\mathrm{free}_{\sP}^{\sO}A} \otimes_\phi {\mathrm{free}_{\sP}^{\sO}A} \to \mathrm{free}_{\sP}^{\sO}A
\]
is isomorphic to the composite
\[
\colim_{\mathcal{E}} F_A {\otimes_\phi}  \colim_{\mathcal{E}} F_A \simeq \colim_{\mathcal{E}\times \mathcal{E}} F_A \otimes_\phi F_A \simeq \colim_{\mathcal{E}\times \mathcal{E}} F_A(-\otimes_\phi -) \to \colim_{\mathcal{E}} F_A,
\]
where the final map is that induced on colimits by the functor $\otimes_\phi\colon \mathcal{E}\times \mathcal{E} \to \mathcal{E}$.  
\end{remark}

\begin{proof}
We will show that $\colim F_A$ has the required universal property. Therefore fix a $\sO$-algebra $B$. First we combine Lemma \ref{univ_prop_colim} and Lemma \ref{univ-prop-slice} to conclude that the mapping space $\Map_{\Alg_{/\sO}(\sC)}(\colim F_A,B)$ is equivalent to the space of lax $\sO$-monoidal lifts 
\[
\begin{tikzcd}
	& {\sC_{/B}} \\
	\Env_{\sO}(\sP) & \sC.
	\arrow["F_A"', from=2-1, to=2-2]
	\arrow["\widetilde{F}_A", dotted, from=2-1, to=1-2]
	\arrow[from=1-2, to=2-2]
\end{tikzcd}
\] 
As noted in Remark \ref{rmk_tilde_F_strong}, $\widetilde{F}_A$ is automatically strong monoidal. Therefore diagrams as above in fact live in $\Mon_{\sO}$, and we may apply the universal property of $\Env_{\sO}$ to conclude that the space of such diagrams is equivalent to the space of lifts in the following diagram 
\[
\begin{tikzcd}
& {\sC_{/B}} \\
\sP & \sC.
\arrow["A"', from=2-1, to=2-2]
\arrow["\widetilde{A}", dotted, from=2-1, to=1-2]
\arrow[from=1-2, to=2-2]
\end{tikzcd}
\] Finally we apply Lemma \ref{univ-prop-slice} to conclude that the space of such diagrams is equivalent to $\Map_{\Alg_{\sP/\sO}(\sC)}(A, \mathrm{fgt}_{\sP}^{\sO} B)$.
\end{proof}

\subsection{Examples}

We briefly give some examples of this theorem.

\begin{example}
We may specialize to the case of $\sP\to \sO$ being the map $\Triv\times_{\E_\infty} \sO \to \sO$ of operads. In this case $\Alg_{\Triv(\sO)\times_{\E_\infty} \sO}(\sC)\simeq \sC$, and so the previous theorem gives conditions for the existence of a left adjoint to the forgetful functor 
\[
\Alg_{\sO}(\sC) \to \sC
\] from $\sO$-algebras in $\sC$ to $\sC$. Suppose for simplicity that $\sO$ is a single colored unital operad. We deduce that for suitable $\sC$, the free $\sO$-algebra functor 
\[\mathrm{Free}_\sO\colon \sC \simeq \Alg_{(\sO \times_{\E_\infty} \Triv)}(\sC)\rightarrow \Alg_{\sO}(\sC)\] is computed via the formula
\[\mathrm{Free}_\sO(X) = \coprod (\sO(n)\times X^{\otimes n})_{h\Sigma_n}.\] This recovers the usual formula for free algebras, see for example \cite[Example 3.1.3.14]{HA}.
\end{example}

\begin{example}
Consider the standard inclusion $\E_0 \to \E_1$ and recall from Example \ref{ex: env-e0-e1} that $\Env_{\E_1}(\E_0) = \Delta_{+}^{\mathrm{inj}}$. An $\E_0$-algebra in a monoidal category $\sC$ is nothing more that an object $X$ of $\sC$ together with a map $\mathbf{1} \to X$, where $\mathbf{1}$ denotes the unit of $\sC$.
An $\E_0$-algebra $\mathbf{1} \to X$ determines a monoidal functor
\[
F_X \colon \Delta_{+}^{\mathrm{inj}} \to \sC
\]
that sends $[n]$ to $X^{\otimes n}$ and an injection $f \colon [n] \to [m]$ to the map $X^{\otimes n} \to X^{\otimes m}$ that, informally speaking, sends the $i$-th factor identically to the $f(i)$-th factor, and is given by the map $\mathbf{1} \to X$ on all other factors.

The theorem tells us that if the tensor product on $\sC$ commutes with $\Delta_+^{\mathrm{inj}}$-shaped colimits in both variables, the free $\E_1$-algebra on an $\E_0$-algebra $X$ in $\sC$ is given by the formula
\[
\mathrm{Free}^{\E_1}_{\E_0} X \simeq \colim_{[n] \in \Delta_+^{\mathrm{inj}}} X^{\otimes n}.
\]

This for instance holds if $\sC$ is the category of pointed spaces with the cartesian product.
In this case, the colimit can be identified with an object from classical homotopy theory: the \emph{James construction} on $X$. 

Another example where this condition holds is the category of spectral Lie algebras; it is shown in \cite{BH} that the cartesian product on this category commutes with colimits indexed by weakly contractible categories in both variables.
It follows that the free $\E_1$-monoid on a spectral Lie algebra is given by the James construction as well.
It is also proven in op. cit. that for any spectral Lie algebra $X$, the free $\E_1$-monoid on $X$ is given by $\Omega \Sigma X$, without connectivity assumptions on $X$, so that we can conclude that
\[
\Omega \Sigma X \simeq \colim_{[n] \in \Delta^{\mathrm{inj}}_{+}} X^{\times n}.
\]
\end{example}

\begin{example}
Let $\E_0 \to \E_\infty$ be the standard inclusion. 
By Example \ref{ex: env-e0-com} we have $\Env_{\Com}(\E_0) = \Fin^{\mathrm{inj}}$. Any $X \in \Alg_{\E_0}(\sC)$ determines a symmetric monoidal functor
\[
F_X \colon \Fin^{\mathrm{inj}} \to \sC,
\]
given by sending a set $T$ to the tensor $X^{\otimes T}$, and sending an injection $f\colon T \to S$ to the inclusion $X^{\otimes T} \to X^{\otimes S}$ determined by $f$.

We now find that
\[
\mathrm{free}_{\E_0}^{\E_\infty} X \simeq \colim_{T \in \Fin^{\mathrm{inj}}} X^{\otimes T}.
\]
Once again when $\sC$ is the symmetric monoidal category of spaces with the cartesian product this colimit has been studied before: it is called the homotopy infinite symmetric product on $X$, denoted $\mathrm{SP_h}(X)$. Recall that the free commutative group on $X$ is computed as $\Omega^\infty \Sigma^\infty X$, and moreover that when $X$ is connected this agrees with the free commutative monoid on $X$. We have therefore recovered the theorem of Schlichtkrull \cite{Sch}, which states that for connected $X$,
\[
\Omega^\infty \Sigma^\infty X \simeq \mathrm{SP_h}(X).
\]

Again, the same result holds in the category of spectral Lie algebras without any connectivity assumptions.
\end{example}

\section{Colimits of \texorpdfstring{$\sO$}{O}-algebras}

As another application of \Cref{thm: free-algebra-colim-env}, we may provide formulas for colimits in $\Alg_{\sO}$.

\subsection{Colimits as free algebras}

Given a category $K$, there is an operad $K^{\sqcup}$ called the \emph{cocartesian operad} on $K$ such that: 
\begin{itemize}
    \item The underlying category of $K^{\sqcup}$ is $K$.
    \item Given objects $x_1, \ldots, x_n, y \in K$, there is an equivalence of spaces 
    \[
    \mathrm{Mul}_{K^{\sqcup}}(\{x_i\}, y) \simeq \prod_{i = 1}^n \Map_K(x_i, y),
    \]
    where $\mathrm{Mul}_{K^{\sqcup}}(\{x_i\}, y)$ denotes the space of multimorphisms in $K^{\sqcup}$.
\end{itemize}

See \cite[§2.4.3]{HA} for more details.

For an operad $\sO$, we write $\sO_K$ for the operad $K^{\sqcup} \times \sO$, where the product is taken in $\Op$.
By projecting to the second factor, we obtain a map of operads $\sO_K \to \sO$.
For any $\sO$-monoidal category $\sC$, it follows from \cite[Theorem 2.4.3.18]{HA} that there is an equivalence of categories
\begin{equation}\label{eq: OK-universal-prop}
\Alg_{\sO_K}(\sC) \simeq \Fun(K, \Alg_\sO(\sC)).
\end{equation}
In other words, an $\sO_K$-algebra is the same as a $K$-indexed diagram of $\sO$-algebras.
We will not distinguish between a functor $K \to \Alg_\sO(\sC)$ and the corresponding $\sO_K$-algebra.

\begin{prop} \label{prop: colim-as-free-algebra}
    Let $K$ be a category, $\sO$ an operad and $\sC$ an $\sO$-monoidal category.
    Given a diagram $f \colon K \to \Alg_\sO(\sC)$, there is an equivalence
    \[
    \mathrm{free}_{\sO_K}^{\sO}(f) \simeq \colim f,
    \]
    where the free functor is computed along the canonical map $p \colon \sO_K \to \sO$.
\end{prop}

\begin{remark}
This proposition should be read in the following way:
if one of $\mathrm{free}_{\sO_K}^{\sO}(f)$ and $\colim f$ exists, the other also exists and they are equivalent.
\end{remark}

\begin{proof}
    It suffices to show that if $A \in \Alg_{\sO}(\sC)$, then $p^*A$ is the constant $K$-indexed diagram at $A$.
    By \cite[Theorem 2.4.3.18]{HA}, the equivalence \cref{eq: OK-universal-prop} is given by restricting along a functor $K \times \sO^\otimes \to \sO_K^\otimes$ described in \cite[Example 2.4.3.5]{HA}.
    Therefore, the functor corresponding to $p^*A$ is given by the composite
    \[
    \begin{tikzcd}
        K \times \sO^\otimes \ar[r] & \sO_K^\otimes \ar[r, "p"] & \sO^\otimes \ar[r, "A"] & \sC^\otimes.
    \end{tikzcd}
    \]
    By definition the composite of the first two morphisms is the projection onto $\sO^\otimes$, and so this is indeed the constant diagram at $A$. 
\end{proof}

Combined with our previous results, this yields explicit formulas for computing colimits in $\Alg_\sO(\sC)$ in terms of colimits in $\sC$.

\begin{corollary} \label{cor: colims-via-envelopes}
    Suppose that $\sC$ is $\sO$-monoidally compatible with $\Env_{\sO}(\sO_K)$-shaped colimits.
    Then every $K$-shaped diagram of $\sO$-algebras in $\sC$ admits a colimit.
    Moreover, the colimit functor is given by
    \[
    \Fun(K, \Alg_{\sO}(\sC)) \simeq \Fun^{\sO}(\Env_{\sO}(\sO_K), \sC) \xrightarrow{\mathrm{colim}} \Alg_\sO(\sC).
    \]
\end{corollary}
\begin{proof}
    This follows from \cref{thm: free-algebra-colim-env} and \cref{prop: colim-as-free-algebra}.
\end{proof}

\subsection{Examples of colimits in $\Alg_{\sO}(\sC)$}
We now sample some applications of  \cref{cor: colims-via-envelopes} by reproving a number of key results from \cite{HA}. We begin with a general observation:

\begin{prop}
Suppose that $\sC$ is an $\sO$-monoidal category that is compatible with all colimits. Then $\Alg_{\sO}(\sC)$ is cocomplete.
\end{prop}

\begin{proof}
We may apply \Cref{prop: colim-as-free-algebra} for any diagram shape to construct arbitrary colimits in $\Alg_{\sO}(\sC)$.
\end{proof}

In specific situations our result also provides formulas for colimits in $\Alg_{\sO}(\sC)$.

\begin{prop}
    Let $\sC$ be a symmetric monoidal category such that the tensor product commutes with geometric realizations.
    Let $B \leftarrow A \rightarrow C$ be a span of commutative algebras in $\sC$.
    Then the pushout is given by $B \otimes_A C$.
\end{prop}

\begin{proof}
    Let $K$ be the span category $ b \leftarrow a \rightarrow c$.
    Taking $\sO = \E_\infty$, we get $\sO_K = K^\sqcup$.
    The category $\Env(K^\sqcup)$ can be described as follows: objects may be identified with words $b^k a^l c^m$ such that $k, l, m \in \Z_{\geq 0}$.
    A morphism from $b^k a^l c^m \to b^x a^y c^z$ is given by a function $f \colon \langle k + l + m\rangle^\circ \to \langle x + y + z \rangle^\circ$ of finite sets such that $f(i) \leq x$ if $i \leq k$ and $f(i) > x + y$ if $i > k + l$.
    The monoidal product is given by adding exponents.
    
    There is an obvious functor $\Delta^{\op} \to \Env(K^\sqcup)$ sending $[n]$ to $ba^nc$. An easy application of Quillen's Theorem A shows that it is final, proving the proposition.
\end{proof}

\begin{prop}
Suppose $\sC$ is an $\sO$-monoidal category and suppose that the tensor product $\otimes_f\colon \sC_{X_1}\times\dots \sC_{X_n} \to \sC_Y$ associated to any active morphism $f\in \mathrm{Mul}_{\sO}(\{X_i\},Y)$ commutes with sifted colimits\footnote{or equivalently that it commutes with sifted colimits in each variable.}. Then the category of $\sO$-algebras admits sifted colimits and they are preserved by the forgetful functor $\mathrm{fgt}\colon \Alg_{\sO}(\sC)\to \sC$.
\end{prop}

\begin{proof}
We will prove the proposition in case $\sO$ has a single color. 
The general case is proved completely analogously, but involves more notation.

Let $S$ be a sifted category.
We will show that the canonical map $S \to \Env_\sO(\sO_S)$ is final.
For this, we have to prove that for every object $x \in \Env_\sO(\sO_S)$ the slice category $x / S$ is contractible.
The object $x$ can be identified with a tuple $(\phi ; x_1, \ldots, x_n)$ where $\phi \in \sO(n)$ and $x_1, \ldots, x_n \in S$.
By inspection the slice $x / S$ is equivalent to the category $S \times_{S^{n}} S^{n}_{(x_1, \ldots, x_n)/}$, where the map $S \to S^n$ is the diagonal. Since $S$ is sifted, its diagonal is final, implying that this category is weakly contractible by Quillen's Theorem A.
\end{proof}

Recall that given an $\sO$-monoidal category, $\Alg_{\sO}(\sC)$ is again an $\sO$-monoidal category such that the forgetful functor $\Alg_{\sO}(\sC)\to \sC$ is $\sO$-monoidal, see \cite[Example 3.2.4.4]{HA}. As our final example we prove \cite[Proposition 5.1.2.9]{HA} in the special case that $k = \infty$.

\begin{prop}
Suppose $\sC$ is a symmetric monoidal category that is compatible with all colimits. Then the symmetric monoidal category $\Alg_{\E_{\infty}}(\sC)$ is compatible with all colimits on weakly contractible indexing categories.
\end{prop}

\begin{proof}
Let $K$ be a weakly contractible category and let $G\colon K\to \Alg_{\E_{\infty}}(\sC)$ be a functor. We will write write $\sE_K$ for $\Env_{\E_{\infty}}((\E_{\infty})_K)$. By comparing universal properties one sees that $\sE_K$ is precisely the finite coproduct completion of $K$ equipped with the cocartesian symmetric monoidal structure. Also let $F_G\colon \sE_K \to \sC$ denote the unique $\E_{k'}$-monoidal functor associated to $G$. Now note that since $K$ is weakly contractible, there is an equivalence $A \simeq \colim_K \underline{A}$, where $\underline{A}$ denotes the constant functor on $A$. Therefore 
\[
(\operatornamewithlimits{colim}_{K}G)\otimes A \simeq \operatornamewithlimits{colim}_K G\otimes\operatornamewithlimits{colim}_K \underline{A}
\] in $\Alg_{\E_{\infty}}(\sC)$. By \cref{prop: colim-as-free-algebra} this is computed in $\sC$ as 
\[
\operatornamewithlimits{colim}_{\sE_K} F_G\otimes \operatornamewithlimits{colim}_{\sE_K} F_{\underline{A}} \simeq \operatornamewithlimits{colim}_{\sE_K\times \sE_K} F_G \otimes F_{\underline{A}}.
\]
Since $\sE_K$ has finite coproducts, the diagonal $\sE_{K} \to \sE_K\times \sE_K$ is a right adjoint and so final. Moreover, by Dunn additivity, the tensor product ${- \otimes - }\colon \sC\times\sC\to \sC$ is a symmetric monoidal functor and therefore so is $F_G \otimes F_{\underline{A}}$. It follows that the restriction of $F_G \otimes F_{\underline{A}}$ to $\sE_{K}$ is induced by its restriction along $K \to K\times K \to \sE_K\times \sE_K$, which is clearly $G\otimes \underline{A}$. Combining these two observations we we obtain further equivalences
\[
\operatornamewithlimits{colim}_{\sE_K\times \sE_K} F_G \otimes F_{\underline{A}} \simeq \operatornamewithlimits{colim}_{\sE_K} F_G \otimes F_{\underline{A}}\simeq \operatornamewithlimits{colim}_{\sE_{K}} F_{G \otimes \underline{A}}.
\]
By applying \cref{prop: colim-as-free-algebra} again this final expression computes $\colim_{K}(G\otimes \underline{A})$ in $\Alg_{\E_{k'}}(\sC)$, and so we have proven the proposition.
\end{proof}

\bibliographystyle{amsalpha}

\end{document}